\documentclass[12pt,fleqn,a4paper]{amsart}

\newcommand\Ethirtyseven{\cite{KnLoc}*{$\mathrm{(5.2)}$}}
\newcommand\Parity{Corollary~2.4}
\newcommand\typeAcolorInequal{Proposition~6.5}
\newcommand\NumberColors{Proposition~2.1}

\usepackage{amsfonts, amssymb}
\usepackage[colorlinks, breaklinks, allcolors=blue]{hyperref} 
\usepackage{amsthm}
\usepackage{breakurl}

\usepackage[alphabetic,nobysame]{amsrefs}

\catcode`\@=11

\def\@settitle{%
  \baselineskip14\p@\relax
    {\Large\bfseries%\fontsize{24}{18}\selectfont%
  \@title}}

\def\@setauthors{%
  \begingroup
  \def\thanks{\protect\thanks@warning}%
  \trivlist
  \footnotesize \@topsep45\p@\relax
  \advance\@topsep by -\baselineskip
  \item\relax
  \author@andify\authors
  \def\\{\protect\linebreak}%
  {\sc\fontsize{12}{10}\selectfont\authors}%
  \ifx\@empty\contribs
  \else
    ,\penalty-3 \space \@setcontribs
    \@closetoccontribs
  \fi
  \endtrivlist
  \endgroup
}

\def\@secnumfont{\bfseries}%

\def\section{\@startsection{section}{1}%
  \z@{.7\linespacing\@plus\linespacing}{.5\linespacing}%
  {\normalfont\bf}}

\renewcommand{\BibLabel}{%
    \Hy@raisedlink{\hyper@anchorstart{cite.\CurrentBib}\hyper@anchorend}%
    [\thebib]\hfill%
}

\DefineSimpleKey{bib}{arxiv}

\BibSpec{article}{%
    +{}  {\PrintAuthors}                {author}
    +{,} { \textit}                     {title}
    +{.} { }                            {part}
    +{:} { \textit}                     {subtitle}
    +{,} { \PrintContributions}         {contribution}
    +{.} { \PrintPartials}              {partial}
    +{,} { }                            {journal}
    +{}  { \textbf}                     {volume}
    +{}  { \PrintDatePV}                {date}
    +{,} { \issuetext}                  {number}
    +{,} { \eprintpages}                {pages}
    +{,} { }                            {status}
    +{,} { \PrintDOI}                   {doi}
    +{,} { \arxiv}           {arxiv}
    +{}  { \parenthesize}               {language}
    +{}  { \PrintTranslation}           {translation}
    +{;} { \PrintReprint}               {reprint}
    +{.} { }                            {note}
    +{.} {}                             {transition}
    +{}  {\SentenceSpace \PrintReviews} {review}
}

\newcommand{\arxiv}[1]{\href{http://arxiv.org/abs/#1}{{\tt arxiv:\hspace{0pt}#1}}}

\catcode`\@=12
\numberwithin{equation}{section}
\swapnumbers

\usepackage[utf8]{inputenc}
\usepackage{parskip}
\usepackage[in]{fullpage} 
\usepackage{tikz}
\usetikzlibrary{calc}
\usepackage{array}
\usepackage{aliascnt}
\usepackage{enumitem}
\setlist[enumerate,1]{label=\textit{\alph*)},ref=\textit{\alph*})}
\setlist[enumerate,2]{label=\textit{\roman*)},ref=\textit{\roman*})}

\theoremstyle{plain}

\newtheorem{theorem}{Theorem}[section]

\newaliascnt{lemma}{theorem}  
\newtheorem{lemma}[lemma]{Lemma}  
\aliascntresetthe{lemma}

\newaliascnt{corollary}{theorem}  
\newtheorem{corollary}[corollary]{Corollary}  
\aliascntresetthe{corollary}

\newtheorem{proposition}[theorem]{Proposition}

\theoremstyle{definition}
\newtheorem{definition}[theorem]{Definition}
\newtheorem*{example}{Example}

\newtheorem*{remark}{Remark}
\newtheorem*{remarks}{Remarks}

\usepackage[capitalize]{cleveref}
\usepackage{microtype}

\newcommand\cC{{\mathcal C}}
\newcommand\cV{{\mathcal V}}
\def\|#1|{\operatorname{#1}}

\newcommand\QQ{{\mathbb Q}}

\newcommand\ZZ{{\mathbb Z}}
\renewcommand\a{\alpha}
\renewcommand\b{\beta}
\newcommand\e{\varepsilon}
\newcommand\<{\langle}
\renewcommand\>{\rangle}
\newcommand\auf{\twoheadrightarrow}
\newcommand\into{\hookrightarrow}
\newcommand\Cq{{\overline C}}
\newcommand\Gq{{\overline G}}
\newcommand\Xq{{\overline X}}
\newcommand\Yq{{\overline Y}}
\newcommand\HS{{\widetilde H}}
\newcommand\pfeil{\rightarrow}

\newcommand\Gm{{\mathbf G_m}}
\newcommand\A{{\mathsf A}}
\newcommand\B{{\mathsf B}}
\newcommand\C{{\mathsf C}}
\newcommand\D{{\mathsf D}}

\newcommand\F{{\mathsf F}}
\newcommand\G{{\mathsf G}}
\renewcommand\P{{\mathbf P}}
\renewcommand\phi{\varphi}
\newcommand\GL{{\operatorname{GL}}}
\newcommand\PGL{{\operatorname{PGL}}}
\newcommand\SL{{\operatorname{SL}}}
\renewcommand\O{{\operatorname{O}}}
\newcommand\SO{{\operatorname{SO}}}
\newcommand\PSO{{\operatorname{PSO}}}
\newcommand\Spin{{\operatorname{Spin}}}
\newcommand\Sp{{\operatorname{Sp}}}
\newcommand\PSp{{\operatorname{PSp}}}
\newcommand\semidir{\ltimes}
\newcommand\leer{\varnothing}

\def\tP/{{\ifmmode (p)\else$(p)$\fi}}
\def\tB/{{\ifmmode (b)\else$(b)$\fi}}
\def\tA/{{\ifmmode (a)\else$(a)$\fi}}
\def\tAA/{{\ifmmode (2a)\else$(2a)$\fi}}

\begin{document}

\title{Spherical roots of spherical varieties}

\author{Friedrich Knop\newline FAU Erlangen-Nürnberg} \address
{Department Mathematik, Emmy-Noether-Zentrum, FAU Erlangen-Nürnberg,
  Cauerstraße 11, 91058 Erlangen, Germany}
\email{friedrich.knop@fau.de}

\begin{abstract}

  Brion proved that the valuation cone of a complex spherical variety
  is a fundamental domain for a finite reflection group, called the
  little Weyl group.  The principal goal of this paper is to
  generalize this theorem to fields of characteristic unequal to
  $2$. We also prove a weaker version which holds in characteristic 2,
  as well. Our main tool is a generalization of Akhiezer's
  classification of spherical varieties of rank $1$.

\end{abstract}

\subjclass[2010]{14M27,14L30,14G17}

\keywords{Spherical varieties, spherical roots, homogeneous
varieties, fields of positive characteristic}

\maketitle

\section{Introduction}
\label{sec:Introduction}

Let $G$ be a connected reductive group defined over an algebraically
closed field $k$ of arbitrary characteristic $p$. A $G$-variety $X$ is
\emph{spherical} if the Borel subgroup $B$ of $G$ has an open orbit in
$X$. For $p=0$ there exists a well-developed structure theory for
spherical varieties. The present paper is part of a program to
generalize this structure theory to arbitrary characteristic.

A crucial part of characteristic zero theory depends on Akhiezer's
list, \cite{Akh}, of spherical varieties of rank 1. In the companion
paper \cite{KnLoc} we compiled results which can be proved without
such a list. In the present paper we use Akhiezer's list, after
generalizing it to arbitrary characteristic, to prove Brion's theorem
on the structure of the valuation cone.

More precisely, for $p=0$, Brion proved in \cite{BriES} that the
valuation cone of a spherical variety is the fundamental domain of a
finite reflection group, the little Weyl group $W_X$ of $X$. Following
Brion, we define $W_X$ to be the group generated by the reflections
about the codimension-1-faces of the valuation cone $\cV(X)$. Our
first important result, \cref{thm:WXfinite}, states that $W_X$ is
always finite. This entails immediately that $\cV(X)$ is always a
union of Weyl chambers of $W_X$.

Next, we investigate when $\cV(X)$ consists of just one chamber. Since
counterexamples were known by Schalke \cite{Schalke} for $p=2$ it came
a bit as a surprise that $\cV(X)$ is indeed a single Weyl chamber
whenever $p\ne2$ (\cref{cor:FundDomain}).

There is even a version of Brion's theorem which is valid for
arbitrary, possibly non-spherical, $G$-varieties. In that case, one
considers the set $\cV_0(X)$ of $G$-invariant valuations of $k(X)$
which are trivial on the subfield $k(X)^B$. Then again, $\cV_0(X)$ is
the fundamental domain for a unique reflection group $W_X$, provided
that $p\ne2$. As a matter of fact, this statement is a formal
consequence of the spherical case (this was already observed in
\cite{InvBew}).

Back to the spherical case: in characteristic zero there are, besides
Brion's, several approaches to the little Weyl group. Unfortunately,
they all use Lie algebra techniques which do not carry over to
positive characteristic. Also, it seems to be hard to make Brion's
original proof work for $p\ne0$. But still, we follow his proof in
``spirit'' in that we carefully investigate the dihedral angles of
$\cV(X)$ ant that we use case-by-case arguments.

More precisely, we study the normal vectors to the codimension-1-faces
of $\cV(X)$. Properly normalized they are called the \emph{spherical
  roots} of $X$. This is now where the rank-1-varieties come in: they
provide us with all possible spherical roots. For $p=0$ their
classification was achieved by Akhiezer, \cite{Akh}. For arbitrary $p$
we follow mostly a simplification due to Brion, \cite{BriR1}. It turns
out (see the table~\S\ref{sec:TABLE}) that there are no surprises: all
cases are known from characteristic zero or can be reduced to a known
one using an inseparable isogeny.

Using this table it is not difficult to see that the angle of any two
spherical roots of $X$ is almost always obtuse. Then, for ruling out
the few exceptions, we use the structure theory developed in
\cite{KnLoc}. For $p=2$, some cases remain which are all listed in
\cref{thm:obtuse}.

\medskip

\noindent{\bfseries Acknowledgment:} I would like to thank Guido
Pezzini for many discussions on the matter of this paper.

\medskip

\noindent{\bfseries Notation:} In the entire paper, the ground field
$k$ is algebraically closed. Its characteristic exponent is denoted by
$p$, i.e., $p=1$ if $\|char|k=0$ and $p=\|char|k$, otherwise. The
group $G$ is connected reductive, $B\subseteq G$ is a Borel subgroup,
and $T\subseteq B$ is a maximal torus. Let $\Xi(T)=\Xi(B)$ be its
character group. The set of simple roots with respect to $B$ is
denoted by $S\subset\Xi(T)$. For $\a\in S$ let $P_\a\subseteq G$ be
the corresponding minimal parabolic subgroup. The Weyl group of $G$
with respect to $T$ is $W$.

For a spherical variety $X$ let $\Xi(X)\subseteq \Xi(T)$ be the group
of weights of $B$-semiinvariant rational functions on $X$. By
definition, the rank of $X$ is $\|rk|X:=\|rk|\Xi(X)$. We also use the
variants $\Xi_\QQ(X)=\Xi(X)\otimes\QQ$ and
$\Xi_p(X)=\Xi(X)\otimes\ZZ_p$ with $\ZZ_p:=\ZZ[\frac1p]$. For any root
$\a$ of $G$ let $\a^r$ be the linear function
$\chi\mapsto\<\chi,\a^\vee\>$ restricted to $\Xi_\QQ(X)$.

A subgroup $H\subseteq G$ is called spherical if $X=G/H$ is
spherical. In that case, we call $\|rk|X$ the \emph{corank} of $H$

The set $\cV(X)$ of $G$-invariant valuations of $X$ can be considered
as a subset of the dual space $N_\QQ(X)=\|Hom|(\Xi_\QQ(X),\QQ)$ (see
\cite{Hyd}*{Cor.~1.8}). It is known to be a finitely generated convex
cone (\cite{Hyd}*{Cor.~5.3}). A \emph{spherical root of $X$} is a
primitive element $\sigma\in\Xi_p(X)\cap\ZZ S$ such that $\sigma$ is
non-positive on $\cV(X)$ and $\cV(X)\cap\{\sigma=0\}$ is one of its
codimension-1-faces. The set of spherical roots is denoted by
$\Sigma(X)$.

Let $Bx_0\subseteq X$ be the open $B$-orbit. The irreducible
components of $Gx_0\setminus Bx_0$ are called the \emph{colors of
  $X$}. We say that $\a\in S$ is of type \tP/ if $P_\a x_0=Bx_0$. The
set of simple roots $\a$ which are of type \tP/ is denoted by
$S^\tP/(X)$ or simply $S^\tP/$.

\section{Classification of spherical varieties of rank one}
\label{sec:Rk1Classification}

The aim of this section is to state the classification of (reduced)
spherical subgroups $H\subseteq G$ with $\|rk|G/H=1$. Thereby, we get
also a list of all possible spherical roots. In characteristic zero,
this has been first achieved by Akhiezer \cite{Akh}. Here, we follow
closely a simplification due to Brion~\cite{BriR1}.

We start by describing the process of (parabolic) induction. For that
let $P\subseteq G$ be a parabolic subgroup, $G_0$ a connected
reductive group and $\pi:P\auf G_0$ a surjective homomorphism. Let
$X_0$ be a $G_0$-variety. Via $\pi$, one may consider $X_0$ as a
$P$-variety. Then $X=G\times^PX_0$ is called the \emph{$G$-variety
  induced from the $G_0$-variety $X_0$ (via $\pi$)}. In practice, it
is convenient to induce from a parabolic subgroup ${}^-\!P$ which is
opposite to the chosen Borel subgroup $B$. The homomorphism from
${}^-\!P$ to $G_0$ is still denoted by $\pi$.

The homomorphism $\pi$ factors through the Levi subgroup
$L={}^-\!P/{}^-\!P_u$. Put $T_0=\pi(T)$. Then $\pi$ induces an
inclusion $\pi^*:\Xi(T_0)\into\Xi(T)$. We call $\pi$ \emph{central} if
$\pi^*(S(G_0))\subseteq S(L)$. If $\pi$ is smooth then it is
central. Thus, in characteristic zero, $\pi$ is always central.  We
may always choose $\pi$ to be central, if we so wish: just replace
$G_0$ by $\Gq_0:={}^-\!P/(\|ker|\pi)^{\|red|}$.

The following is well-known in characteristic zero:

\begin{lemma}
  \label{lem:Induction}

  Let $X$ be the $G$-variety induced from the $G_0$-variety $X_0$.

  \begin{enumerate}

  \item\label{it:Induction.a} $X$ is spherical if and only if $X_0$ is
    spherical.

  \end{enumerate}

  Assume this. Then:

  \begin{enumerate}[resume]

  \item\label{it:Induction.b} $\Xi(X)=\Xi(X_0)$. In particular,
    $\|rk|X=\|rk|X_0$.

  \item\label{it:Induction.c} $\cV(X)=\cV(X_0)$.

  \end{enumerate}

  Assume, additionally, that $\pi$ is central. Then

  \begin{enumerate}[resume]

  \item\label{it:Induction.d} $\Sigma(X)=\Sigma(X_0)$.

  \item\label{it:Induction.e} $S^\tP/(X)=S^\tP/(X_0)\cup
    S(\|ker|\pi)$.

  \end{enumerate}

\end{lemma}

\begin{proof}

  Assertions \ref{it:Induction.a} and \ref{it:Induction.b} follow from
  the fact that $X$ contains $P\times^LX_0=P_u\times X_0$ as a
  $B$-invariant open subset.

  For the rest of the assertions we may replace $X_0$ by its open
  $G_0$-orbit and therefore assume that $X_0$ and $X$ are
  homogeous. Then, up to a positive scalar, the elements of $\cV(X)$
  correspond to smooth embeddings $X\into\Xq$ such that
  $D=\Xq\setminus X$ is a homogeneous divisor. The canonical morphism
  $X\pfeil G/{}^-\!P$ extends to $\Xq$. Thus, $\Xq$ is of the form
  $G\times^{{}^-\!P}\Xq_0$ where $\Xq_0\setminus X_0$ is a homogeneous
  divisor, as well. Hence it corresponds to an element of
  $\cV(X_0)$. This easily implies \ref{it:Induction.c}.  Assertion
  \ref{it:Induction.d} is now a direct consequence of
  \ref{it:Induction.c} and the definition of spherical roots.

  Finally, assume $P_\a$, $\a\in S(G)$ stabilizes the open $B$-orbit
  $X_1$ in $X$. Then its image $P_u$ in $G/{}^-\!P$ is stabilized, as
  well. This shows $\a\in S(L)$. In that case, $P_\a$ stabilizes $X_1$
  if and only if $\pi(P_\a)\subseteq G_0$ stabilizes the open
  $B_0$-orbit in $X_0$. Assertion \ref{it:Induction.e} follows.
\end{proof}

Of particular importance is the case when $X_0=G_0/H_0$ is
homogeneous. Then $X=G/H$ is homogeneous, as well, with
$H=\pi^{-1}(H_0)$. If $H$ cannot be obtained by induction in a
non-trivial way, i.e., with $\|dim|G_0<\|dim|G$, then it is called
\emph{cuspidal}. Thus, this means two things:

\begin{enumerate}

\item The only parabolic $P\subseteq G$ with $P_u\subseteq H\subseteq
  P$ is $P=G$.

\item The only connected normal subgroup $K$ of $G$ with $K\subseteq
  H$ is $K=1$.

\end{enumerate}

\noindent Observe that cuspidality is preserved under isogenies. More
precisely, image and preimage of a cuspidal subgroup under an isogeny
are cuspidal.

Any finite subgroup $H$ of the 1-dimensional torus $G=\Gm$ is
certainly cuspidal. Spherical varieties of rank one which are
parabolically induced from a torus are called \emph{horospherical}. In
other words, a spherical subgroup $H\subset G$ of corank $1$ is
horospherical if it is of the form $\|ker|\chi$ where $\chi$ is a
non-trivial character of a parabolic subgroup of $G$. We are going to
prove later:

\begin{lemma}\label{lem:Rk1Semisimple}

  Assume that $G$ contains a cuspidal spherical subgroup of corank
  $1$. Then either $G\cong\Gm$ or $G$ is semisimple.

\end{lemma}

\noindent Hence, it suffices to treat the case that $G$ is
semisimple. In the following it is convenient to assume that $G$ is
even of adjoint type which is obviously not a big loss of
generality. The main classification theorem is now:

\begin{theorem}\label{thm:MainClassifThm}

  Let $G$ be a semisimple group of adjoint type and let $H\subset G$
  be a cuspidal spherical subgroup of corank $1$. Then:

  \begin{enumerate}

  \item\label{it:MainClassifThm.a}

    The pair $(G,H)$ appears in the table~\S\ref{sec:TABLE}.

  \item\label{it:MainClassifThm.b}

    The coefficients of the spherical root $\sigma\in\Sigma(G/H)$ are
    indicated in the column with caption ``$\sigma$''. The set
    $S^\tP/(G/H)$ is denoted in the form of black dots.

  \end{enumerate}

\end{theorem}

\begin{remark}

  The table is a bit condensed in the sense that an entry $\<s\>\cdot
  H_0$ denotes the two groups $H=H_0$ and its normalizer
  $H=\<s\>H_0$. In characteristic $p\ne2$, the spherical roots for
  these groups differ by a factor of $2$ which is also indicated.

\end{remark}

Before we proceed to the proof of \cref{thm:MainClassifThm} we
discuss, following Luna, the notion of an ``abstract spherical root''
of $G$:

\begin{definition}

  \begin{enumerate}

  \item A \emph{spherical root of $G$} is an element $\sigma\in\ZZ S$
    such that there is a spherical $G$-variety $X$ of rank $1$ with
    $\Sigma(X)=\{\sigma\}$. The set of spherical roots of $G$ is
    denoted by $\Sigma(G)$.

  \item A spherical root $\sigma$ of $G$ is \emph{compatible to
      $S'\subseteq S$} if there is a spherical $G$-variety $X$ of rank
    $1$ with $\Sigma(X)=\{\sigma\}$ and $S^\tP/(X)=S'$.

  \end{enumerate}

\end{definition}

We introduce the \emph{support} $|\sigma|\subseteq S$ of a spherical
root $\sigma$ as the smallest set of simple roots needed to express
it. In other words,
\begin{equation}
  |\sigma|:=\{\a\in S\mid n_\a\ne0\}
\end{equation}
when
\begin{equation}
  \sigma=\sum_{\a\in S}n_\a \a\quad\hbox{with}\ n_\a=n_\a(\sigma)\in\ZZ_{\ge0}.
\end{equation}
Now, the following recipe on how to determine all spherical roots and
all compatible sets follows directly from \cref{lem:Induction}.

\begin{enumerate}

\item $\sigma\in\Sigma(G)$ if and only if $(|\sigma|,\sigma)$ appears
  in table~\S\ref{sec:TABLE}.

\item $\sigma$ and $S'\subseteq S$ are compatible if and only if all
  $\a\in S'\setminus|\sigma|$ are orthogonal to $\sigma$ and
  $|\sigma|\cap S'$ consists of exactly the ``black vertices'' in the
  diagram of $\sigma$ in table~\S\ref{sec:TABLE}.

\end{enumerate}

\begin{example}

  The spherical roots for the root system $\A_3$ are:
  \begin{align*}
    &\a_1,\a_2,\a_3,\a_1+\a_2,\a_2+\a_3,\a_1+\a_2+\a_3\\
    &2\a_1,2\a_2,2\a_3\text{ for }p\ne2\\
    &\a_1+2\a_2+\a_3,\a_1+\a_3\\
    &\a_1+q\a_3,q\a_1+\a_3\text{ for }q=p^a>1
  \end{align*}
  The roots in the last row are new to positive characteristic. In
  particular, as opposed to characteristic 0, a fixed group may have
  infinitely many abstract spherical roots.

\end{example}

\section{Proof of the classification}
\label{sec:ProofOfClassification}

This section is devoted to the proof of \cref{thm:MainClassifThm}. Let
$X=G/H$ be spherical of rank $1$. Then $N_\QQ(X)\cong\QQ$ and $\cV(X)$
is either equal to $N_\QQ(X)$ (in case $X$ is horospherical) or a
half\-line (otherwise). Thus, it follows from the theory of spherical
embeddings (see \cite{Hyd}) that $X$ admits a unique smooth
equivariant completion $\Xq$ such that $\partial X:=\Xq\setminus X$ is
pure of codimension one and consists of either two (horospherical
case) or one (otherwise) homogeneous components. The embedding $\Xq$
is maximal in the following sense: Let $X'$ be a complete normal
$G$-variety and $\phi:X'\pfeil\Xq$ a birational $G$-equivariant
morphism. Then $\phi$ is an isomorphism.

\begin{lemma}\label{lem:intermediate}

  Let $H\subset G$ be a spherical subgroup of corank $1$ and let
  $H\subseteq K\subseteq G$. Then either $K/H$ or $G/K$ is a complete
  variety.

\end{lemma}

\begin{proof}[Proof {\textrm(same as \cite{BriR1}*{1.3~Lemme})}]

  Assume $Y:=K/H$ is not a complete variety. Then it is not closed in
  $\Xq$. Let $\Yq$ denote its closure. Then $G\times^K\Yq\pfeil\Xq$ is
  an isomorphism, by maximality. Thus, by inverting this morphism, we
  obtain an equivariant morphism $\Xq\auf G/K$ which implies that
  $G/K$ is a complete variety.
\end{proof}

A first application is the

\begin{proof}[Proof of \cref{lem:Rk1Semisimple}]

  Let $H\subset G$ be spherical of corank $1$, let $Z\subseteq G$ be
  the connected center of $G$, and put $K:=ZH$. Then $H$ is normal in
  $K$ with $K/H=Z/(Z\cap H)$ a torus. Thus, if $G/K$ is complete then
  $K$ is a parabolic in $G$ and $H$ is induced from $K/H$. Cuspidality
  implies $G=K/H\cong\Gm$. If, on the other hand, $K/H$ is complete
  then $H=K$, i.e., $Z\subseteq H$. But then $Z=1$ since $H$ is
  cuspidal, i.e., $G$ is semisimple.
\end{proof}

From here to the end of this section, $G$ will denote a semisimple
group of adjoint type (unless stated otherwise). Moreover, $H\subset
G$ is spherical of corank $1$. The classification forks then into two
cases:

\begin{lemma}\label{lem:fork}

  Let $H\subset G$ be cuspidal. Then $H$ is reductive or there is a
  parabolic subgroup $P\subset G$ with $P=HP_u$.

\end{lemma}

\begin{proof}[Proof \rm(similar to \cite{BriR1}*{1.3 Th\'eor\`eme})]

  By a theorem of Borel-Tits \cite{BorTits}*{3.1}, there is a
  parabolic subgroup $P\subseteq G$ with $H\subseteq P$ and
  $H_u\subseteq P_u$. Choose $P$ minimal with this property and put
  $K=HP_u$. Then either $G/K$ or $K/H$ is complete.

  In the first case, $K$ is parabolic in $G$. Since $H\subseteq
  K\subseteq P$ and $H_u\subseteq H_uP_u=K_u$ we get $K=P$ by
  minimality of $P$.

  In the second case, $K/H=P_u/(H\cap P_u)$ is affine, connected, and
  complete, hence trivial. This and cuspidality imply $P=G$ and
  $H_u=P_u=1$, i.e., $H$ is reductive.
\end{proof}

Before we proceed we transfer a result of Dynkin
\cite{Dyn}*{15.1~Theorem} for semisimple Lie algebras in
characteristic $0$ to semisimple groups in arbitrary
characteristic. The proof is basically the same.

\begin{lemma}\label{lem:Dynkin}

  Let $G$ be a semisimple group with $G_1,\ldots, G_s$ its simple
  normal subgroups. Let $H\subset G$ be maximal among connected
  reductive proper subgroups. Then either

  \begin{enumerate}

  \item

    $H=H_i\prod_{j\ne i}G_j$ for some $i$ where $H_i\subset G_i$ is a
    maximal connected reductive proper subgroup or

  \item

    $H=H_{ij}\prod_{l\ne i,j}G_l$ for some $i\ne j$ where
    $H_{ij}\subset G_iG_j$ is a connected subgroup such that the
    projections $H_{ij}\pfeil G_i/(G_i\cap G_j)$ and $H_{ij}\pfeil
    G_j/(G_i\cap G_j)$ are isogenies.

  \end{enumerate}

\end{lemma}

\begin{proof}

  Without loss of generality we may assume $G=G_1\times\ldots\times
  G_s$ with $s\ge2$. Suppose that one of the projections $H\pfeil G_i$
  is not surjective and denote its image by $H_i$. Then $H$ is
  contained in, hence is equal to $H_i\prod_{j\ne i}G_j$.

  Now assume that all projections $H\pfeil G_i$ are surjective. Then
  for every $i$ there is a simple normal subgroup $N\subseteq H$ such
  that the composition $N\into H\into G\auf G_i$ is an isogeny. All
  other factors of $H$ are then mapped to the trivial group in
  $G_i$. Furthermore, there must be one factor $N$ of $H$ which is
  being mapped surjectively to more than one factor, say $G_i$ and
  $G_j$, of $G$ since otherwise $H=G$. Put $H_{ij}=N$. Thus, $H$ is
  contained in, hence is equal to $H_{ij}\prod_{l\ne i,j}G_l$.
\end{proof}

\begin{lemma}

  Assume $G$ is not simple and that $H$ is cuspidal and
  reductive. Then $G/H$ is isomorphic to
  $(\PGL(2)\times\PGL(2))/(\|id|\times F_q)\PGL(2)$ where $F_q$ is a
  Frobenius morphism.

\end{lemma}

\begin{proof}[Proof \rm(similar to \cite{BriR1}*{2.2 Corollaire})]

  Assume first that $H$ is connected. Then $H$ is maximal among all
  connected reductive subgroups of $G$
  (\cref{lem:intermediate}). Since $H$, being cuspidal, does not
  contain any simple factor of $G$, \cref{lem:Dynkin} implies that
  $G$, being of adjoint type, is the direct product of two simple
  factors, $G=G_1\times G_2$, and that $\phi_i:H\pfeil G_i$ is an
  isogeny. Thus, there is a finite morphism $H\times H/\Delta(H)\pfeil
  G/H$ which implies that $H$, as an $H\times H$-variety, is of rank
  $1$. But this rank equals the rank of $H$ as a group. We conclude
  that the $G_i$ are semisimple, of rank one, and of adjoint type and
  therefore isomorphic to $\PGL(2)$. Any isogeny from $\SL(2)$ to
  $\PGL(2)$ contains the (schematic) center of $\SL(2)$ in its kernel.
  Thus, since $\phi_1\times\phi_2$ is an embedding we see that also
  $H\cong \PGL(2)$, as well. Moreover, $\phi_1$ and $\phi_2$ must be
  (conjugate to) powers of the Frobenius morphism and one of them is
  the identity. Thus, up to a switch of factors, we have
  $G=\PGL(2)\times\PGL(2)$ and $H=(\|id|\times F_q)\PGL(2)$.

  Finally, let $H$ be not necessarily connected. Then
  $H^0=(\|id|\times F_q)\PGL(2)$ and $H^0\subseteq H\subseteq
  N_G(H^0)=H^0$, i.e., $H=H^0$.
\end{proof}

\begin{lemma}

  Assume that $G$ is simple and $H$ is cuspidal and reductive. Then
  $(G,H)$ appears in the table~\S\ref{sec:TABLE}.

\end{lemma}

\begin{proof}

  Assume first, that $H$ is connected. Then we refrain from
  generalizing the arguments of Brion in \cite{BriR1}*{2.3,
    2.4}. Instead, with \cite{KnRoe} we have now a classification of
  \emph{all} connected spherical reductive sub\-groups of simple
  groups at our disposal. This generalizes Krämer's
  classification~\cite{Krae} in characteristic zero. In a nutshell,
  the outcome of \cite{KnRoe} is that, up to an isogeny of $G$, there
  appears only one more case in positive characteristic, namely
  $H=\G_2\times\SL(2)\subset G=\Sp(8)$ in $\|char|k=2$. But that case
  has rank $3$ by \cite{KnRoe}*{Prop.~4.5}. As a result, all pairs
  $H\subset G$ appear in Krämer's list and therefore must appear in
  Akhiezer's list \cite{Akh}, as well (always up to an isogeny of
  $G$).

  If $H$ is non-connected, then it is a subgroup of the normalizer
  $N=N_G(H^0)$ of its connected component. Thus we have compute the
  normalizers of all connected subgroups in the list. The result will
  be that $H^0$ is of index at most two in its normalizer. This
  implies $H=N$ whenever $N\ne H^0$. These cases are indicated by the
  presence of $\langle s_\alpha\rangle$.

  There are three main cases to consider. The first are those cases
  which are special to characteristic $2$ and $3$. For all of them
  there is a bijective isogeny to a case defined over $\ZZ$. The same
  holds for $\PGL(2)$ embedded diagonally into $\PGL(2)\times
  \PGL(2)$. Clearly, the normalizers are not affected. This way, we may
  assume that $H\subset G$ is defined over $\ZZ$.

  Then we consider the equal rank case, i.e., where $H$ contains a
  maximal torus $T$ of $G$. In that case, every element of $N/H$ is
  represented by an element of the Weyl group, which normalizes the
  root system of $H$. This means that $N/H$ is independent of the
  characteristic. Thus we may just copy the results from
  characteristic zero.

  In the remaining cases, all automorphisms
  of $H$ are inner. This implies $N=CH$ where $C$ is the centralizer
  of $H$ in $G$. We argue that $C=1$ and therefore $N=H$ in all
  cases.

  Now we discuss the cases separately. First, let $H=\G_2$ in
  $G=\SO(7)$. If $p\ne2$, the representation of $H$ on $k^7$ is
  irreducible. Thus $C$ consists only of scalars and therefore
  $C=1$. For $p=2$, the pair $H\subset G$ is isogenous to
  $\G_2\subset\PSp(6)$. Again the irreducibility of the $\G_2$-module
  $k^6$ implies $C=1$.

  All other cases are of the form $H=\SO(2n-1)$ in $G=\PSO(2n)$ with
  $n\ge2$ (the case $n=2$ and $n=3$ are the same as
  $\PGL(2)\subset\PGL(2)^2$ and $\PSp(4)\subset\PGL(4)$,
  respectively). It suffices to show that $\tilde C$, the centralizer
  of $\SO(2n-1)$ in $\SO(2n)$ consists only of scalars. Let
  $u\in\tilde C$ be unipotent and let $v_0\in k^{2n}$ be the vector
  fixed by $\SO(2n-1)$. Since $v_0$ is unique up to a scalar, it
  follows that $uv_0=v_0$ and therefore $u\in\SO(2n-1)$. Because the
  center of $\SO(2n-1)$ is trivial, we conclude $u=1$. Thus, $\tilde
  C$ consists only of semisimple elements. If $p=2$ this implies
  $\tilde C=1$ since $k^{2n}$ is indecomposable as
  $\SO(2n-1)$-representation. If $p\ne2$ then $k^{2n}=k\oplus
  k^{2n-1}$. Thus, if $s\in\tilde C$ then
  $s=\operatorname{diag}(\lambda,\mu,\ldots,\mu)$ with
  $\lambda^2=\mu^2=\lambda\mu^{2n-1}=1$. This implies
  $\lambda=\mu\in\{\pm1\}$. Thus $s$ is a scalar.
\end{proof}

According to \cref{lem:fork}, the last remaining step in the
classification is:

\begin{lemma}

  Assume that $H$ is cuspidal and that there is a parabolic subgroup
  $P\subset G$ with $P=HP_u$. Then $(G,H)$ is either isomorphic to
  $(\PGL(n),\GL(n-1)), n\ge2$ or one of the non-reductive pairs in the
  table~\S\ref{sec:TABLE}.

\end{lemma}

\begin{proof}[Proof \rm(generalization of \cite{BriR1}*{2.1})]

  Assume first that $H$ is connected. Let $L\subseteq P$ be a Levi
  subgroup. Then, by assumption, the homomorphism $H\pfeil P/P_u\cong
  L$ is surjective. Thus, $H_u$ is mapped to a unipotent normal
  subgroup of $L$, i.e., to $1$ which means $H_u\subseteq P_u$. Let
  $Z\subseteq L$ be the center of $L$. It is a torus since $G$ is of
  adjoint type. Hence, it can be lifted to a subtorus $Z'$ of
  $H$. Because $Z'$ is $P$-conjugate to $Z$ we may, without loss of
  generality, assume that $Z=Z'\subseteq H$. But then
  $L=C_P(Z)\subseteq H$ (see, e.g., \cite{BorLAG}*{11.14 Cor. 2}),
  i.e., $H=L\semidir H_u$.

  Recall the maximal compactification $\Xq$ of $X=G/H$ and let let
  $\Yq\subseteq\Xq$ be the closure of $Y=P/H\subseteq G/H$. Then
  $G\times^P\Yq\pfeil\Xq$ is proper and birational, hence an
  isomorphism by the maximality property of $\Xq$. This shows that
  $\Yq$ consists of two $P$-orbits, namely $Y$ and $\partial
  Y=\Yq\setminus Y$. The latter is a complete $P$-orbit of codimension
  $1$ in $\Yq$. In particular, since $P_u$, being solvable, has a
  fixed point in $\partial Y$ it acts in fact trivially and $L$ acts
  transitively on $\partial Y$.

  The 1-dimensional root subgroups $U_\a\subset B$, $\a\in S$,
  generate $U=(B,B)$. Thus, since $H$ is not horospherical, there is a
  simple root $\a$ such that $U_\a\not\subseteq H$. This means
  $U_\a\subseteq P_u$ but $U_\a\not\subseteq H_u$. Then $C:=U_\a
  H/H\subseteq Y$, the $U_\a$-orbit in $P/H$, is an affine curve. It
  is closed, since $Y=P_u/H_u$ is affine and $U_\a$ is unipotent. Now
  let $\Cq$ be its closure in $\Yq$. Let ${}^-\!B\subseteq G$ be the
  Borel subgroup opposite to $B$. Then $U_\a$ is normalized by
  ${}^-\!B\cap L$ which implies that $\Cq$ is ${}^-\!B\cap
  L$-invariant. Because ${}^-\!B\cap L$ is a Borel subgroup of $L$, we
  infer that $L\Cq$ is an irreducible closed $L$-invariant subset of
  $\Yq$. Since $\Cq$ meets $\partial Y$ and $L$ acts transitively on
  $\partial Y$ this implies that $\partial Y\subsetneq L\Cq$. For
  dimension reasons we get $\Yq=L\Cq$ and therefore $Y=LC$. But the
  center $Z\subseteq L$ has only two orbits in $C$, namely $\{0\}$ and
  its complement. This shows that $L$ acts transitively on
  $Y\setminus\{0\}$.

  At this point we can show that it was no loss of generality to
  assume that $H$ is connected. Indeed let $\HS\subseteq P$ be a
  subgroup with $\HS^0=H$. Then $\HS/H$ is a finite $L$-invariant
  subset of $Y$. Hence $\HS/H=\{0\}$ and therefore $\HS=H$.

  Let $Q\subseteq G$ be the parabolic generated by $P$ and
  $U_{-\a}$. Then $U_\a\not\subseteq Q_u$ implies that $\a$ is not a
  weight in $Q_uH_u$. Hence $Q_uH_u/H_u$ is a proper $L$-invariant
  subvariety of $Y=P_u/H_u$ which means $Q_u\subseteq H_u\subseteq
  H\subseteq Q$. Since $H$ is cuspidal we conclude that $Q=G$, i.e.,
  that $P$ is a maximal parabolic subgroup corresponding to the set of
  simple roots $\Sigma'=\Sigma(G)\setminus\{\a\}$. Moreover, the
  cuspidality of $H$ implies that $G$ is simple and that $\Sigma(G)$
  is connected.

  Since $P_u$ acts transitively on $Y$ and $L$ acts transitively on
  $Y\setminus\{0\}$ the action of $P$ on $Y$ is doubly
  transitive. According to \cite{KnMT}*{Satz 2} there is an
  isomorphism $Y\cong{\mathbf A}^n$ such that the $P$-action is given
  by a surjective homomorphism $\pi:P\auf(\Gm\cdot
  G_0)\semidir\mathbf{G}_a^n$ where $\mathbf{G}_a^n$ acts by
  translations, $\Gm$ acts by scalars, and $G_0$ is either
  $\SL(n),n\ge2$ or $\Sp(n), n\ge2\ {\rm even}$ or $\G_2$, with $n=6$
  and $\|char|k=2$. This means in particular that $Y=P_u/H_u$ is a
  linear representation of $L$. It is irreducible with lowest weight
  $\a$.

  We conclude that $\a$ is an ``end'' of $\Sigma$ and $\Sigma'$ is of
  type $\A$ or $\C$ in general or of type $\B$ if
  $\|char|k=2$. Moreover, $-\a$ is a $p$-power multiple of the
  canonical representation (or its dual). Up to isomorphism, this
  leaves precisely the following cases:
  \begin{equation}\label{eq:22}
    (\Sigma,\a)=(\A_n,\a_1), (\B_n,\a_n), (\C_n,\a_1), (\G_2,\a_1)
  \end{equation}
  in arbitrary characteristic and
  \begin{equation}\label{eq:23}
    (\Sigma,\a,p)=(\C_n,\a_n,2), (\B_n,\a_1,2), (\G_2,\a_2,3)
  \end{equation}
  in characteristic $p>0$. Using a non-central isogeny, the latter
  cases reduce to the former.

  Clearly all cases appear (see table). Moreover, $H$ contains the
  maximal torus $T$. Therefore, $H_u$ is generated by all $U_\beta$
  where $\beta$ is a root of $P_u$ which is not a weight of $Y$. This
  shows uniqueness of $H$.
\end{proof}

Finally, we prove part \ref{it:MainClassifThm.b} of
\cref{thm:MainClassifThm}. For this we use the following

\begin{lemma}

  Let $X=G/H$ be defined over a field of characteristic $p>0$ and that
  $X$ can be lifted to a homogeneous varity $X_0$ over a field of
  characteristic $0$. Then $X$ is spherical if and only if $X_0$
  is. Morover, there are equalities $\Xi_p(X)=\Xi_p(X_0)$ and
  $S^\tP/(X)=S^\tP/(X_0)$.
\end{lemma}

\begin{proof}

  By assumption, there is a complete discrete valuation ring $R$ with
  residue field $k$ and uniformizer $\pi\in R$ and a smooth $R$-scheme
  $\mathcal{X}$ which has $X$ and $X_0$ as special fiber and generic
  geometric fiber, respectively. Then the equivalence of $X$ and $X_0$
  being spherical is \cite{KnRoe}*{Thm.~3.4}.

  According to \cite{KnRoe}*{Lem.~3.1}, for every $B$-semininvariant
  $f$ on $X$ there is an $n\ge1$ such that $f^n$ extends to a
  semiinvariant on $\mathcal{X}$. Moreover, \cite{FvdK}*{Prop.~41}, the
  exponent $n$ can be chosen to be a power of $p$. This combined shows
  $\Xi(X)\subseteq\Xi_p(X_0)$.

  Conversely, for let $\chi\in\Xi(X_0)$ let $f$ be a semiinvariant on
  $X_0$ with character $\chi_f=\chi$. After possibly replacing $R$ by
  a (ramified) extension we may assume that $f$ is a rational function
  on $\mathcal{X}$ which has poles at most along $X$. Thus, there is a
  unique exponent $m\in\ZZ$ such that $\tilde f:=\pi^mf$ is regular on
  $\mathcal{X}$. Then, the restriction of $\tilde f$ to $X$ is a
  $B$-semiinvariant with character $\chi$. This even shows
  $\Xi(X_0)\subseteq\Xi(X)$.

  For $\mathcal{X}$ quasiaffine, the equality $S^\tP/(X)=S^\tP/(X_0)$
  follows from $S^\tP/(X)=\{\a\in S\mid \a^r=0\}$. The general case is
  reduced to this by passing to an affine cone over $X$.
\end{proof}

Now let $H\subset G$ be a member of Table \S\ref{sec:TABLE} which can
be lifted to characteristic zero. Then $\Xi(X_0)$ and $S^\tP/(X_0)$
are well known (see, e.g., \cite{Wass}*{Table~1}). From this and the
Lemma it is easy to determine $\Xi_p(X)$ and $S^\tP/(X)$.

All other cases are isogenous to liftable ones. So the corresponding
data can be easily calculated, as well.  This finishes the proof of
\cref{thm:MainClassifThm}.

\section{The structure of the valuation cone}
\label{sec:ValuationCone}

Our main application for classifying spherical roots is to extend
Brion's Theorem \cite{BriES} on the structure of the valuation cone of
a spherical variety.

In the following, we choose an auxiliary $W$-invariant rational scalar
product on $\Xi_\QQ(T)$. For any $0\ne\sigma\in \Xi_\QQ(X)$ let
\begin{equation}
  s_\sigma(\chi)=\chi-\sigma^r(\chi)\sigma,\quad\hbox{with
  }\sigma^r(\chi):=2\frac{(\chi,\sigma)}{(\sigma,\sigma)}
\end{equation}
be the unique orthogonal reflection of $\Xi_\QQ(X)$ with
$s_\sigma(\sigma)=-\sigma$.

\begin{lemma}\label{lem:Wextend}

  Let $X$ be a spherical $G$-variety and $\sigma\in\Sigma(X)$. Then:
  \begin{enumerate}

  \item\label{it:Wextend.a}

    There is $n_\sigma\in W$ with $n_\sigma|_{\Xi_\QQ(X)}=s_\sigma$.

  \item\label{it:Wextend.b}

    Assume $\sigma\not\in 2S$. Then there is a root $\b$ of $G$ with
    $\sigma^r=\b^r$.

  \end{enumerate}

\end{lemma}

\begin{proof}

  We check all items of the table~\S\ref{sec:TABLE}.

  \noindent 1.~case: $\sigma=u\a$ where $\a$ is a root of $G$ and
  $u\in\{1,2\}$. Then clearly $n_\sigma=s_\a$ works for
  \ref{it:Wextend.a}. Moreover, \ref{it:Wextend.b} is trivial for
  $u=1$. So let $u=2$. Then $p\ne2$ and there are three cases to
  consider:

  \begin{itemize}

  \item $|\sigma|=\B_n$ ($n\ge2$), $\a=\a_1+\ldots+\a_n$, and $\a_n\in
    S^\tP/(X)$. Then $\a_n^r=0$. On the other hand
    $\sigma^r=\frac12\a^r=\frac12(2\b^r+\a_n^r)=\b^r$ where
    $\b=\a_1+\ldots+\a_{n-1}$ is a root.

  \item $|\sigma|=\G_2$, $\a=2\a_1+\a_2$, and $\a_2\in S^\tP/$. Then
    $\sigma^r=\frac12\a^r=\frac12(2\a_1^r+3\a_2^r)=\a_1^r$.

  \item $|\sigma|=\G_2$, $\a=3\a_1+2\a_2$, $\a_1\in S^\tP/$, and
    $p=3$. Then $\sigma^r=\frac12\a^r=\frac12(\a_1^r+2\a_2^r)=\a_2^r$.

  \end{itemize}

  For the other cases we first prove \ref{it:Wextend.a}. For this, we
  use the following observation: let $\sigma=u\a+v\b$ with
  $u,v\in\QQ_{>0}$ and $\a$, $\b$ orthogonal roots of $G$. Assume
  $u^{-1}\a^r-v^{-1}\b^r=0$. Then $s_\sigma=s_\a
  s_\b|_{\Xi_\QQ(X)}$. Indeed, the assumptions imply
  $\sigma^r=u^{-1}\a^r=v^{-1}\b^r$. Hence
  \begin{equation}
    s_\a s_\b(\chi)=\chi-\a^r(\chi)\a-\b^r(\chi)\b=
    \chi-\sigma^r(\chi)\sigma=s_\sigma(\chi)
  \end{equation}
  for all $\chi\in\Xi_\QQ(X)$.

  \medskip

  \noindent 2.~case: $\sigma=\a_1+q\a_2$ with $\a_1,\a_2\in S$
  orthogonal. Then $\a_1^r-q^{-1}\a_2^r=0$ by \Ethirtyseven.

  \medskip

  \noindent 3.~case: In the remaining four cases, we claim that there
  is a decomposition $\sigma=u\a+v\b$ where $u,v\in\QQ_{>0}$ and $\a$,
  $\b$ are orthogonal roots of $G$ with
  $u^{-1}\a^\vee-v^{-1}\beta^\vee\in\<\gamma^\vee\mid\gamma\in
  S^\tP/\>_\QQ$. Indeed:
$$
\begin{array}{lll}

  |\sigma|=\A_3&\sigma=\a_1+2\a_2+\a_3=\a+\b&=\e_1+\e_2-\e_3-\e_4,\\
  &\a=\a_1+\a_2&=\e_1-\e_3,\\
  &\beta=\a_2+\a_3&=\e_2-\e_4,\\
  &\a^\vee-\beta^\vee=\a_1^\vee-\a_3^\vee&=(\e_1-\e_2)-(\e_3-\e_4)\\
  |\sigma|=\B_3&\sigma=\a_1+2\a_2+3\a_3=\a+\b&=\e_1+\e_2+\e_3,\\
  &\a=\a_1+\a_2+2\a_3&=\e_1+\e_3,\\
  &\beta=\a_2+\a_3&=\e_2,\\
  &\a^\vee-\beta^\vee=\a_1^\vee-\a_2^\vee&=(\e_1-\e_2)-(\e_2-\e_3)\\
  |\sigma|=\D_n&\sigma=2\a_1+\ldots+2\a_{n-2}+\a_{n-1}+\a_n=\a+\b&=2\e_1,\\
  &\a=\a_1+\ldots+\a_{n-2}+\a_{n-1}&=\e_1-\e_n,\\
  &\beta=\a_1+\ldots+\a_{n-2}+\a_n&=\e_1+\e_n,\\
  &\a^\vee-\beta^\vee=\a_{n-1}^\vee-\a_n^\vee&=(\e_{n-1}-\e_n)-(\e_{n-1}+\e_n)\\
  |\sigma|=\C_3&\sigma=2\a_1+4\a_2+3\a_3=2\a+\b&=2\e_1+2\e_2+2\e_3\\
  (p=2)&\a=\a_1+\a_2+\a_3&=\e_1+\e_3,\\
  &\b=2\a_2+\a_3&=2\e_2,\\
  &\frac12\a^\vee-\b^\vee=\frac12\a_1^\vee-\frac12\a_2^\vee&=\frac12(\e_1-\e_2)-\frac12(\e_2-\e_3).\\

\end{array}
$$
This shows \ref{it:Wextend.a} in all cases. For \ref{it:Wextend.b} use
$\sigma^r=u^{-1}\a^r=v^{-1}\b^r$ and the fact that in each case one of
$u$ or $v$ is $1$.
\end{proof}

\begin{remarks}

  \begin{enumerate}

  \item The proof shows that $n_\sigma$ can be chosen to be either a
    simple reflection $s_\a$ or $n_\sigma=s_\a s_\b$ where $\a$ and
    $\b$ are \emph{very orthogonal roots} meaning that there is $w\in
    W$ such that $w\a$ and $w\b$ are orthogonal simple roots.

  \item The element $n_\sigma$ can be chosen independently of the
    choice of the scalar product on $\Xi(T)$. Hence, also $s_\sigma$
    is independent of the scalar product.

  \end{enumerate}
\end{remarks}

\begin{definition}

  The subgroup of $W_X\subseteq\GL(\Xi_\QQ(X))$ generated by all
  $s_\sigma$, $\sigma\in\Sigma(X)$, is called the \emph{little Weyl
    group of $X$}.

\end{definition}

\begin{theorem}\label{thm:WXfinite}

  Let $X$ be a spherical variety. Then:

  \begin{enumerate}

  \item\label{it:WXfinite.a} The little Weyl group $W_X$ is finite.

  \item\label{it:WXfinite.b} The groups $\Xi_p(X)$ and $\ZZ S\cap
    \Xi_p(X)$ are $W_X$-invariant.

  \item\label{it:WXfinite.c} The set $R_X:=W_X\Sigma(X)$ is a (finite)
    root system with Weyl group $W_X$.

  \end{enumerate}
\end{theorem}

\begin{proof} \ref{it:WXfinite.a} The little Weyl group is finite
  since, by \cref{lem:Wextend}\ref{it:Wextend.a}, it is a subquotient
  of $W$.

  \ref{it:WXfinite.b} Let $\sigma\in\Sigma(X)$. Then
  $\sigma\in\Xi_p(X)$. Moreover, it follows from
  \cite{KnLoc}*{\Parity} (for $\sigma\in2S$) and
  \cref{lem:Wextend}\ref{it:Wextend.b} (for $\sigma\not\in2S$) that
  $\sigma^r$ takes values in $\ZZ_p$ on $\Xi_p(X)$. This combined
  implies the $s_\sigma$-invariance of $\Xi_p(X)$. Then the invariance
  of $\ZZ S\cap \Xi_p(X)$ follows from the $W$-invariance of $\ZZ S$
  and \cref{lem:Wextend}\ref{it:Wextend.a}.

  \ref{it:WXfinite.c} follows from the fact that $R_X$ consists of
  primitive vectors in the $W_X$-invariant lattice $\ZZ S\cap
  \Xi_p(X)$.
\end{proof}

\begin{remark}

  The group $\Xi(X)$ itself is, in general, not $W_X$-invariant. Take
  for example $G=\GL(2)$ and $X=\mathbf{A}^2\times{}^{(q)}\!\P^1$.
  Here $^{(q)}\P^1$ denotes the projective line with the Frobenius
  twisted $G$-action. Then $\Xi(X)=\ZZ q\e_1\oplus\ZZ\e_2$ which is
  not $s_\a$-stable unless $q=1$.

\end{remark}

\begin{corollary}

  The valuation cone $\cV(X)$ of a spherical variety $X$ is the union
  of Weyl chambers for the little Weyl group $W_X$.

\end{corollary}

\begin{remark}

  Schalke, \cite{Schalke}, has shown that if $p=2$ and
  $X=\PGL(3)/SO(3)$ then $\Sigma(X)=\{\a_1,\a_1+\a_2\}$. Thus $\cV(X)$
  can be identified with the set of rational triples $(x_1,x_2,x_3)$
  with $x_1+x_2+x_3=0$ and $x_1\le x_2,x_3$ which is the union of the
  two chambers $\{x_1\le x_2\le x_3\}$ and $\{x_1\le x_3\le x_2\}$.

\end{remark}

The next goal is now to prove that $\cV(X)$ is actually just one Weyl
chamber, provided that $p\ne2$.  A constraint on the angles between
spherical roots is given by the following theorem. Its proof is
deferred to \cref{sec:Angle}.

\begin{theorem}\label{thm:obtuse}

  Let $X$ be a spherical variety and let $\sigma,\tau\in\Sigma(X)$ be
  distinct spherical roots with $(\sigma,\tau)>0$. Then $p=2$ and, up
  to a switch of $\sigma$ and $\tau$, the triple $|\sigma|\cup|\tau|$,
  $\sigma$, $\tau$ is contained in the following table.
  \begin{equation}\label{eq:32}
    \begin{array}{lll}

      |\sigma|\cup|\tau|&\sigma&\tau\\

      \noalign{\smallskip\hrule\smallskip}

      \A_2&\a_1&\a_1+\a_2\\

      \B_n, n\ge2&\a_2+\ldots+\a_n&\a_1+2\a_2+\ldots+2\a_n\\

      \C_n, n\ge2&2\a_2+\ldots+2\a_{n-1}+\a_n&\a_1+2\a_2+\ldots+2\a_{n-1}+\a_n\\

      \G_2&\a_2&\a_1+\a_2

    \end{array}
  \end{equation}

\end{theorem}

\begin{remark}

  In characteristic 2 one can show that all exceptional cases in
  \cref{thm:obtuse} do occur. Namely:

  \begin{itemize}

  \item $\Sigma(X)=\{\a_1,\a_1+\a_2\}\subset\A_2$ where
    $X=\SL(3)/\SO(3)$ (see \cite{Schalke}).

  \item
    $\Sigma(X)=\{\a_2+\ldots+\a_n,\a_1+2\a_2+\ldots+2\a_n\}\subset\B_n$,
    $n\ge2$ where $X$ is the open $\SO(2n+1)$-orbit in the
    Grassmannian of codimension-2-subspaces in $k^{2n+1}$ (the
    defining representation).

  \item $\Sigma(X)=\{2\a_2+\ldots+2\a_{n-1}+\a_n,
    \a_1+2\a_2+\ldots+2\a_{n-1}+\a_n\}\subset\C_n$, $n\ge2$, is
    isogenous to the preceding case, thus occurs as well. More
    concretely: the variety $\Sp(2n)/\O(2n)$ is isomorphic to an
    affine space ${\mathbf A}^{2n}$ on which $\Sp(2n)$ acts by affine
    linear transformations. Now $X$ is the open orbit in the set of
    affine lines of ${\mathbf A}^{2n}$.

  \item $\Sigma(X)=\{\a_2,\a_1+\a_2\}\subset\G_2$ where $X=G/H$
    and\newline $H=\SO(3)\cdot
    U_{2\a_1+\a_2}U_{3\a_1+\a_2}U_{3\a_1+2\a_2}\subset
    P_{\a_2}\subset\G_2$.

  \end{itemize}
Details will appear elsewhere.

\end{remark}

From \cref{thm:obtuse} we derive the main result of this paper:

\begin{theorem}

  Let $X$ be a spherical variety which is defined over a field of
  characteristic $p\ne2$. Then $\Sigma(X)$ is a system of simple roots
  for the root system $R_X=W_X\Sigma(X)$.

\end{theorem}

\begin{proof}

  According to \cref{thm:obtuse}, the angle between any two spherical
  roots is $\ge\frac{\pi}{2}$. Hence \cite{Bou}*{Ch.~5, \S3, no.~5,
    Lemme~3} (or \cite{BriES}*{Th\'eor\`eme~3.1}) implies that
  $\Sigma(X)$ is linearly independent. Now we argue as in the proof of
  \cite{BriES}*{Th\'eor\`eme~3.5}.
\end{proof}

The following corollaries are immediate consequences:

\begin{corollary}\label{cor:FundDomain}

  Let $p\ne2$. Then the valuation cone $\cV(X)$ of a spherical variety
  $X$ is a Weyl chamber for the little Weyl group $W_X$.

\end{corollary}

\begin{corollary}

  Let $p\ne2$ and $X$ a spherical variety. Then the set $\Sigma(X)$ of
  spherical roots is linearly independent. This means, in particular,
  that the valuation cone $\cV(X)$ is a cosimplicial cone.

\end{corollary}

The theory of spherical embeddings, \cite{Hyd}, provides us with the
following relaxed verion of a wonderful embedding:

\begin{corollary}

  Let $p\ne2$ and let $X=G/H$ be a homogeneous spherical variety such
  that $H^{\|red|}$ is of finite index in its normalizer. Then there
  is an equivariant normal completion $X\into\Xq$ with the following
  properties:

  \begin{enumerate}

  \item\label{it:WondVar.a} The boundary $\Xq\setminus X$ is the union
    of $r=\|rk|X$ irreducible divisors $D_1,\ldots,D_r$.

  \item\label{it:WondVar.b} The map $I\mapsto\left(\bigcap_{i\in
        I}D_i\right)^{\|red|}$ is a bijection between subsets of
    $\{1,\ldots,r\}$ and the set of orbit closures in $\Xq$.

  \end{enumerate}
\end{corollary}

\begin{proof}The condition on the normalizer implies that the
  valuation cone $\cV(X)$ is pointed (\cite{Hyd}*{Thm.~6.1}). Thus,
  $\cC=\cV(X)$ defines a toroidal embedding $X\into\Xq$. Properties
  \ref{it:WondVar.a} and \ref{it:WondVar.b} follow from the fact that
  $\cV(X)$ is a simplicial cone.
\end{proof}

\begin{remark}

  Observe, that, as opposed to wonderful varieties, we made no claims
  of smoothness or transversality.

\end{remark}

\Cref{thm:obtuse} has also the following consequence:

\begin{corollary}

  Every (internal) dihedral angle of the valuation cone equals either
  \begin{equation}
    \textstyle\frac{1}{6}\pi,\ \frac{1}{4}\pi,\ \frac{1}{3}\pi,\ \frac12\pi,
    \quad \frac{2}{3}\pi,\ \frac{3}{4}\pi,\ \hbox{or}\  \frac{5}{6}\pi. 
  \end{equation}
  The last three values occur only for $p=2$.

\end{corollary}

\section{The angles between spherical roots}\label{sec:Angle}

This section is devoted to the proof of \cref{thm:obtuse}. We are
going to use the notation
\begin{equation}\label{eq:25}
  |\sigma|_p:=|\sigma|\cap S^\tP/,\quad
  |\sigma|_s:=|\sigma|\setminus S^\tP/
\end{equation}
which we call the \emph{parabolic support} and the \emph{singular
  support} of $\sigma$. An inspection of table~\S\ref{sec:TABLE} shows
that the singular support consists of either 1 or 2 elements.

We start with a trivial but useful observation, the saturation
principle:

\begin{lemma}

  Let $\sigma$ be a spherical root, let $\a\in|\sigma|$, and let
  $\b\in S^\tP/$ be connected to $\a$ in the Dynkin diagram (i.e.,
  $(\a,\b)<0$). Then $\b\in|\sigma|$, as well.

\end{lemma}

\begin{proof}

  Since $\a$ and $\b$ are connected we have $(\a,\b)<0$. Now suppose
  $\b\not\in|\sigma|$. Then
  \begin{equation}\label{eq:26}
    0=(\b,\sigma)=n_\a(\b,\a)+\sum_{\gamma\in|\sigma|\setminus\{\a\}}n_\gamma(\b,\gamma)<0,
  \end{equation}
  a contradiction.
\end{proof}

First we show that the angle between almost any two spherical roots is
automatically obtuse, just for combinatorial reasons. The length of a
root does not play a r\^ole for that. Therefore, we are considering
only spherical roots $\sigma$ which are \emph{reduced}, i.e., for
which $\frac12\sigma$ is not a spherical root.

\begin{lemma}

  Let $\sigma$ and $\tau$ be two distinct reduced spherical roots of
  $G$ which are compatible with some subset $S^\tP/\subseteq
  S$. Assume that both $\sigma$ and $\tau$ have connected support and
  that $(\sigma,\tau)>0$. Then, up to a switch of $\sigma$ and $\tau$,
  the triple $|\sigma|\cup|\tau|,\sigma,\tau$ is contained in the
  following table:
  \begin{equation}\label{eq:27}
    \begin{array}{lll}

      |\sigma|\cup|\tau|&\sigma&\tau\\

      \noalign{\smallskip\hrule\smallskip}

      \A_2&\a_1&\a_1+\a_2\\

      \B_2&\a_1&\a_1+\a_2\\

      \G_2&\a_2&\a_1+\a_2\\

      \noalign{\smallskip\hrule\smallskip}

      \multicolumn{3}{l}{\text{Additionally for $p=2$:}}\\

      \B_2&\a_1+\a_2&\a_1+2\a_2\\

      \B_n, n\ge2&\a_2+\ldots+\a_{n-1}+\a_n&\a_1+2\a_2+\ldots+2\a_{n-1}+2\a_n\\

      \C_n, n\ge3&2\a_2+\ldots+2\a_{n-1}+\a_n&\a_1+2\a_2+\ldots+2\a_{n-1}+\a_n\\

      \noalign{\smallskip\hrule\smallskip}

      \multicolumn{3}{l}{\text{Additionally for $p=3$:}}\\

      \G_2&\a_1&3\a_1+\a_2

    \end{array}  
  \end{equation}

\end{lemma}

\begin{proof}

  The case $\#|\sigma|\cup|\tau|=2$ can be easily handled case by
  case. These provide for six cases in the table.  Therefore, assume
  from now on that $\#|\sigma|\cup|\tau|>2$. Let $\sigma=\sum_{\a\in
    S}n_\a\a$. Then, because of
  \begin{equation}\label{eq:28}
    0<(\sigma,\tau)=\sum_{\a\in|\sigma|_s}n_\a(\a,\tau)
  \end{equation}
  we have $|\sigma|_s\cap|\tau|_s=|\sigma|_s\cap|\tau|\ne\leer$.

  First, we treat the case that $\#|\sigma|_s=1$. Then
  $|\sigma|_s\subseteq|\tau|$. The saturation principle and the
  connectedness of $|\sigma|$ imply that $|\sigma|\subseteq|\tau|$. If
  also $\#|\tau|_s=1$ then, by symmetry, $|\tau|\subseteq|\sigma|$, as
  well. An inspection of the list of spherical roots shows that this
  is not possible with $\#|\sigma|>2$.

  Now let $|\tau|_s=2$. Then the saturation principle and the table
  leaves the following possibilities (recall $\#|\tau|\ge3$):
  \begin{equation}\label{eq:29}
    \begin{array}{lllr}

      |\tau|&\sigma&\tau&(\sigma,\tau)\\

      \noalign{\smallskip\hrule\smallskip}

      \B_4&\a_2+2\a_3+3\a_4=\e_2+\e_3+\e_4&\a_1+\a_2+\a_3+\a_4=\e_1&=0\\

      \C_n, n\ge2&\a_1=\e_1-\e_2&\a_1+2\a_2+\ldots+\a_n=\e_1{+}\e_2&=0\\

      \noalign{\smallskip}
      \multicolumn{4}{l}{\text{Additionally for $p=2$:}}\\

      \B_3&\a_2+2\a_3=\e_2+\e_3&\a_1+\a_2+\a_3=\e_1&=0\\

      \B_n,n\ge2&\a_1=\e_1-\e_2&\a_1+2\a_2+\ldots+2\a_n=\e_1{+}\e_2&=0\\

      \B_n,n\ge2&\a_2+\ldots+\a_n=\e_2&\a_1+2\a_2+\ldots+2\a_n=\e_1{+}\e_2&>0\\

      \C_n,n\ge2&2\a_2+\ldots+\a_n=2\e_2&\a_1+2\a_2+\ldots+\a_n=\e_1{+}\e_2&>0\\

      \C_3&\a_2+\a_3=\e_2+\e_3&2\a_1+2\a_2+\a_3=2\e_1&=0\\

      \C_4&2\a_2{+}4\a_3{+}3\a_4{=}2\e_2{+}2\e_3{+}2\e_4&2\a_1+2\a_2+2\a_3+\a_4=2\e_1&=0

    \end{array}
  \end{equation}

  Thus, we obtain exactly the two series in the statement.

  Finally, assume that $\#|\sigma|_s=\#|\tau|_s=2$. Then
  $|\sigma|_s=|\tau|_s$ would imply, as above, that $|\sigma|=|\tau|$,
  hence $\sigma=\tau$. This leaves the case $|\sigma|_s=\{\a,\b\}$,
  $|\tau|_s=\{\b,\gamma\}$. We have
  \begin{equation}\label{eq:30}
    0<(\sigma,\tau)=n_\a(\a,\tau)+n_\b(\b,\tau)
  \end{equation}
  which shows that $(\b,\tau)>0$ and, by symmetry, $(\b,\sigma)>0$. If
  one goes through all possible pairs $\b,\sigma$, the saturation
  principle shows that $|\sigma|=\{\a\}\cup T$ and, by symmetry,
  $|\tau|=\{\gamma\}\cup T$ where $T=|\sigma|\cap|\tau|$. Thus, we
  arrive at the following possibilities:
  \begin{equation}\label{eq:31}
    \begin{array}{lll}

      |\sigma|\cup|\tau|&\sigma&\tau\\
      \noalign{\smallskip\hrule\smallskip}
      \D_n, n\ge3&\a_1+\ldots+\a_{n-2}+\a_{n-1}=\e_1-\e_n&\a_1+\ldots+\a_{n-2}+\a_n=\e_1+\e_n\\
      \B_3&\a_1+\a_2=\e_1-\e_3&\a_2+\a_3=\e_2\\
      \noalign{\smallskip}
      \multicolumn3l{\text{Additionally for $p=2$:}}\\
      \C_3&\a_1+\a_2=\e_1-\e_3&2\a_2+\a_3=2\e_2

    \end{array}
  \end{equation}

  In all these cases, we have $(\sigma,\tau)=0$.
\end{proof}

\begin{proof}[Proof of \cref{thm:obtuse}]

  We treat first the case when the support of $\sigma$ is not
  connected. Then $\sigma=\a_1+q\a_2$ with $\a_1,\a_2\in S$
  orthogonal. Equation \Ethirtyseven\ implies that $\tau$ is
  orthogonal to $\lambda:=\a_1^\vee-q^{-1}\a_2^\vee$. It follows that
  $(\a_2,\tau)=x(\a_1,\tau)$ for some $x>0$. Thus,
  $0<(\sigma,\tau)=y(\a_1,\tau)$ with $y=1+qx>0$ which implies
  $\a_1\in|\tau|_s$. By symmetry, also $\a_2\in|\tau|_s$, hence
  $\{\a_1,\a_2\}=|\tau|_s$.  Now one can use the table of spherical
  roots in \S\ref{sec:TABLE} to show that this is only possible if
  $\tau=\a_1+q'\a_2$. But then $(\lambda,\tau)=0$ implies $q'=q$, in
  contradiction to $\sigma\ne\tau$.

  So, up to factors of $2$, the pair $\sigma,\tau$ is contained in
  Table \eqref{eq:27}. Looking at $S^\tP/$, observe that $\tau$ is
  necessarily reduced. Moreover, if $p\ne2$ then $\sigma$ is reduced
  as well because of the parity condition \cite{KnLoc}*{\Parity}.

  To exclude other cases in \eqref{eq:27} we use some results from
  \cite{KnLoc} which require that $\sigma$ and $\tau$ are
  \emph{neighbors in $\Sigma(X)$}. This means that
  $\QQ_{\ge0}\sigma+\QQ_{\ge0}\tau$ is a two-dimensional face of the
  cone $\QQ_{\ge0}\Sigma(X)$. Fortunately, this holds in all cases of
  interest:

  \begin{lemma}\label{lem:neighbor2}

    Let $\sigma,\tau\in\Sigma(X)$ be distinct spherical roots with
    $|\sigma|\subseteq|\tau|$. Then $\sigma$ and $\tau$ are neighbors
    in $\Sigma(X)$.

  \end{lemma}

  \begin{proof}

    Otherwise there is an equation
    \begin{equation}
      u\sigma+v\tau=t_1\eta_1+\ldots+t_s\eta_s
    \end{equation}
    with $u,v,t_1,\ldots,t_s>0$ and
    $\eta_1,\ldots,\eta_s\in\Sigma(X)\setminus\{\sigma,\tau\}$. Clearly
    $s\ge1$ and $|\eta_1|\subseteq|\sigma|\cup|\tau|=|\tau|$. Because
    $\eta_1$ is compatible to $|\tau|_p$ and $\#|\tau|_s\le2$ we infer
    that $\eta_1$ is a linear combination of $\sigma$ and $\tau$. But
    then one of $\sigma$, $\tau$ or $\eta_1$ could not be extremal in
    $\QQ_{\ge0}\Sigma(X)$.
  \end{proof}

  Next, we show that the three cases
  \begin{equation}\label{eq:33}
    (\A_2,\a_1,\a_1+\a_2),\ (\B_2,\a_1,\a_1+\a_2),\ (\G_2,\a_2,\a_1+\a_2)
  \end{equation}
  don't occur when $p\ne2$. For this we use heavily the machinery of
  \cite{KnLoc}. See especially \cite{KnLoc}*{\S2} for unexplained
  notation.

  In all three cases, we have $\sigma\in S^\tA/$ and let $D_1$ and
  $D_2$ be the two colors moved by $\sigma$. Because of
  $\delta_{D_1}^{(\sigma)}+\delta_{D_2}^{(\sigma)}=\sigma^r$, we have
  $\delta_{D_i}(\tau)>0$ for $i=1$ or $2$ in contradiction to
  \cite{KnLoc}*{\typeAcolorInequal}.

  For $p\ne2,3$ we are done. For $p=3$ only one more case has to be
  checked. But, using the exceptional self-isogeny of $\G_2$, that
  case can be reduced to the last case of \eqref{eq:33} which has been
  ruled out before.

  Also for $p=2$ only one more case remains to be checked, namely
  $(\B_2,\a_1+\a_2,\a_1+2\a_2)$. So assume $X=G/H$ where $H$ is a
  connected subgroup of $G=\SO(5)$ with
  $\Sigma(X)=\{\a_1+\a_2,\a_1+2\a_2\}$. Then both simple roots are of
  type \tB/. Thus there are two colors such that $\delta_D$ is
  proportional to $\a_1^\vee$ and $\a_2^\vee$, respectively. Since
  $\cV(X)$ and the two colors can be separated by a hyperplane, the
  space $X$ is affine, \cite{Hyd}*{Thm.~6.7}, and $H$ is
  reductive. Since there are only two colors, $H$ is even semisimple
  (\cite{KnLoc}*{\NumberColors}). Moreover, $\|dim|H=4$ (see
  \cite{Hyd}*{Thm.~6.6}). But such a group does not exist.
\end{proof}

\begin{remark}

  Let $\a\in S\cap\frac12\Sigma(X)$. Then Luna's axiom $\rm(\Sigma1)$
  for a spherical system stipulates that $\<\sigma,\a^\vee\>$ is a
  nonpositive even integer for all
  $\sigma\in\Sigma(X)\setminus\{2\a\}$ (see \cite{LuTypA}*{2.1} or
  \cite{BrLu}*{1.2.1}). The discussion above shows that the
  nonpositivity part is in fact superfluous since it follows from the
  other axioms (more precisely from $\rm(S)$, the parity part of
  $\rm(\Sigma1)$, and $\rm(\Sigma2)$).

\end{remark}

In passing, we proved most parts of the following statement which we
state for future reference:

\begin{proposition}

  Let $\sigma\ne\tau\in\Sigma(X)$ with $|\sigma|\subseteq|\tau|$. Then
  $(|\tau|,\sigma,\tau)$ is contained in the following table:
  \begin{equation}
    \begin{array}{llll}

      |\tau|&&\sigma&\tau\\

      \noalign{\smallskip\hrule\smallskip}

      \B_4&&\a_2+2\a_3+3\a_4&\a_1+\a_2+\a_3+\a_4\\

      \C_n,&n\ge2&\a_1&\a_1+2\a_2+\ldots+\a_n\\

      \C_n,&n\ge2&2\a_1\ (p\ne2)&\a_1+2\a_2+\ldots+\a_n\\

      \G_2&&\a_1&\a_1+\a_2\\

      \noalign{\smallskip}

      \multicolumn4l{\text{Additionally for $p=2$:}}\\

      \A_2&&\a_1&\a_1+\a_2\\

      \B_n,&n\ge2&\a_1&\a_1+2\a_2+\ldots+2\a_n\\

      \B_n,&n\ge2&\a_2+\ldots+\a_n&\a_1+2\a_2+\ldots+2\a_n\\

      \C_n,&n\ge2&2\a_2+\ldots+\a_n&\a_1+2\a_2+\ldots+\a_n\\

      \C_4&&2\a_2+4\a_3+3\a_4&2\a_1+2\a_2+2\a_3+\a_4\\

      \G_2&&\a_2&\a_1+\a_2\\

      \noalign{\smallskip}

      \multicolumn4l{\text{Additionally for $p=3$:}}\\

      \G_2&&\a_2&3\a_1+\a_2

    \end{array}
  \end{equation}

\end{proposition}

\begin{proof}

  Most cases are already contained in table \eqref{eq:29}. Missing
  from that table are those cases with $\#|\tau|=2$ which can be
  easily dealt with by hand. So it remains to be shown that the case
  $(|\tau|,\tau,\sigma)=(\C_3,2\a_1+2\a_2+\a_3,\a_2+\a_3)$ (and the
  isogenous case $(\B_3,\a_1+\a_2+\a_3,\a_2+2\a_3)$) with $p=2$ does
  not occur.

  By \cref{lem:neighbor2} we may localize at $\Sigma$ (see
  \cite{KnLoc}*{6.1}). Thus, it suffices to exclude the existence of a
  connected subgroup $H\subseteq G=\Sp(6)$ with
  $\Sigma(X)=\{\sigma,\tau\}$. There are two colors $D_1$ and $D_2$
  which are both of type \tB/. Hence $\delta_{D_i}$ is a positive
  multiple of $\a_i^r$. The affinity criterion in \cite{Hyd} shows
  that $H$ is reductive. Then \cite{KnLoc}*{2.1} implies that $H$ is
  even semisimple. Moreover, from $S^\tP/=\{\a_2\}$ one computes
  $\|dim|H=11$. Hence, $H$ must be of type $\A_1\times\A_2$. Since
  $\|rk|H=\|rk|G=3$ the root system of $H$ would be a subroot system
  of $G$ which is not the case.
\end{proof}

\begin{remark}

  All cases occur. For the first four see, e.g., \cite{Wass}. The
  others are treated in the remark after \cref{thm:obtuse} or are
  isogenous to one of the previous cases.

\end{remark}

\section{Invariant valuations of arbitrary $G$-varieties}
\label{ArbVar}

For $\|char|k=0$, Brion's Theorem on the valuation cone of spherical
varieties has been used in \cite{InvBew} to obtain generalizations for
arbitrary $G$-varieties. Now that we have similar results at our
disposal we can do the same in arbitrary characteristic.

The idea is to consider everything ``relative'' to
$B$-invariants. More precisely, let $X$ be any $G$-variety and let
$K=k(X)^B$ be its field of $B$-invariant rational functions. Clearly,
$K=k$ means that $X$ is spherical. In general, the transcendence
degree of $K$ over $k$ is called the \emph{complexity of $X$} so that
spherical means complexity 0.

Let $\cV(X)$ be the set of $G$-invariant ($\QQ$-valued) valuations of
the filed $k(X)$. This set plays the same r\^ole in determining
equivariant compactifications of $X$ as in the spherical case (see
\cite{LV}). Because it is, in general, too big to control, we
partition it into manageable pieces by fixing the restriction to
$K$. More precisely, for a valuation $v_0$ of $K$ put
\begin{equation}\label{eq:35}
  \cV_{v_0}(X):=\{v\in\cV(X)\mid v|_K=v_0\}.
\end{equation}
It turns out (see below) that these sets are quite easy to
understand. So if it is possible to control all valuations of $K$ then
it is possible to understand $\cV(X)$ as a whole. This is trivially
the case for $K=k$ (i.e., spherical varieties) but also if $K$ is of
transcendence degree 1 over $k$ (the case of complexity 1).

Of particular importance are \emph{central valuations}, i.e.,
valuations whose restriction to $K$ is trivial. These are precisely
the elements of $\cV_0(X)$. Now consider the group $k(X)^{(B)}$ of
$B$-semiinvariant rational functions on $X$ and let
$\Xi(X)\subseteq\Xi(T)$ be the group of ensuing characters. Then we
obtain a short exact sequence
\begin{equation}\label{eq:36}
  1\pfeil K^*\pfeil k(X)^{(B)}\pfeil\Xi(X)\pfeil0.
\end{equation}
A central valuation $v$ is by definition trivial on $K^*$. Hence it
defines a homomorphism on $\Xi(X)$. This way we get a map
\begin{equation}
  \cV_0(X)\pfeil N^0_\QQ(X):=\|Hom|(\Xi(X),\QQ)
\end{equation}
which turns out to be injective (\cite{InvBew}*{3.6}). It identifies
$\cV_0(X)$ with a finitely generated convex cone
(\cite{InvBew}*{6.5}). The generalization of \cref{cor:FundDomain} is:

\begin{theorem}

  Let $X$ be a $G$-variety defined over a field of characteristic
  $p\ne2$. Then the set $\cV_0(X)$ of central valuations is the
  fundamental domain for a finite reflection group $W_X$ acting on
  $\Xi_p(X)$.

\end{theorem}

\begin{proof}

  Follows from the spherical case in the same way as in characteristic
  zero (\cite{InvBew}*{9.2}). See also the slightly overoptimistic
  remark after the theorem.
\end{proof}

Surprisingly, this determines also the structure of non-central
valuations. They have to satisfy a mild technical condition, though: a
valuation $v$ of $k(X)$ is called \emph{geometric} if there exists a
normal model $\Xq$ of $X$ and an irreducible divisor $D$ in $\Xq$ such
that $v$ is a rational multiple of the induced valuation $v_D$. It is
known that an invariant valuation is geometric if and only if its
restriction to $K$ is geometric (\cite{InvBew}*{4.4}). Since
valuations of fields of transcendence degree $\le1$ are always
geometric, geometricity is no restriction for varieties of complexity
$\le1$.

Now fix a geometric valuation $v_0$ and let $N^{v_0}_\QQ(X)$ be the
set of homomorphisms $v:k(X)^{(B)}\pfeil\QQ$ such that
$v|_{K^*}=v_0$. Any two elements differ by an element of
$N^0_\QQ(X)$. This means that $N^{v_0}_\QQ(X)$ has the structure of an
affine space with $N^0_\QQ(X)$ as group of translations. As before,
the map $\cV_{v_0}(X)\pfeil N^{v_0}_\QQ(X)$ is injective, thereby
identifying $\cV_{v_0}(X)$ with a locally polyhedral convex set.

\begin{theorem}

  Let $X$ be a $G$-variety defined over a field of characteristic
  $p\ne2$ and let $v_0$ be a geometric valuation of $K=k(X)^B$. Then
  $\cV_{v_0}(X)=\tilde v+\cV_0(X)$ for some $\tilde v\in
  \cV_{v_0}(X)$. In particular, there is an action of $W_X$ on
  $N^{v_0}_\QQ(X)$ generated by affine linear reflections such that
  $\cV_{v_0}(X)$ is a fundamental domain for this action.

\end{theorem}

\begin{proof}

  Again, the same proof (\cite{InvBew}*{9.2}) as in characteristic
  zero works. The $W_X$-action is defined by $w(\tilde v+v)=\tilde
  v+wv$ where $v\in N^{v_0}_\QQ(X)$.
\end{proof}

\begin{remark}

  Observe that this action does not depend on the choice of $\tilde v$
  because any two choices differ by a $W_X$-fixed vector.

\end{remark}

Now we glue everything together. For this let
\begin{equation}
  \cV^{\|geom|}(X)=\{v\in\cV(X)\mid v\hbox{ geometric}\}
\end{equation}
Moreover, it can be shown (\cite{InvBew}*{4.3}) that
$\bigcup\limits_{v_0\ {\rm geom.}}N^{v_0}_\QQ(X)$ equals the set
$\cV^{\|geom|}(k(X)^U;T)$ of $T$-invariant geometric valuations of the
field of $U$-invariant rational functions on $X$ (with $U=(B,B)$).

\begin{corollary}

  Let $p\ne2$. Then there is a $W_X$-action on
  $\cV^{\|geom|}(k(X)^U;T)$, acting as an affine linear reflection
  group on each piece $N^{v_0}_\QQ(X)$, such that $\cV^{\|geom|}(X)$
  is a fundamental domain for this action.

\end{corollary}

\newpage

\section{Table of cuspidal spherical varieties of rank 1 for groups of
  adjoint type}
\label{sec:TABLE}

\def\oo(#1,#2;#3){\draw[thick] (#1,#2) node[above]{$\ph#3$} circle
  (3pt);} \def\ph{\vrule width 0pt height 8pt}
\def\xx(#1,#2;#3){\filldraw[thick] (#1,#2) node[above]{$#3$} circle
  (3pt);} \def\ddd(#1,#2){ \draw[thick]($(#1,#2)+(0.08,0)$) --
  ($(#1,#2)+(0.6,0)$); \filldraw[thick] ($(#1,#2)+(0.8,0)$) circle
  (0.3pt); \filldraw[thick] ($(#1,#2)+(1,0)$) circle (0.3pt);
  \filldraw[thick] ($(#1,#2)+(1.2,0)$) circle (0.3pt);
  \draw[thick]($(#1,#2)+(1.4,0)$) -- ($(#1,#2)+(2,0)$);}

\def\scalar{0.7}

$$
\begin{array}{llll}

  G&H&&\sigma, S^\tP/\\
  \noalign{\smallskip\hrule\smallskip}
  \PGL(2)&\<s_{a_1}\>\cdot\GL(1)&&
  \begin{tikzpicture}[scale=\scalar]
    \oo(0,0;1)
  \end{tikzpicture}[\times2]_{p\ne2}\\

  \PGL(n)&\GL(n-1)&n\ge3&
  \begin{tikzpicture}[scale=\scalar]
    \oo(0,0;1) \draw[thick] (0.1,0)--(0.9,0); \xx(1,0;1) \draw[thick]
    (1.1,0)--(1.9,0); \xx(2,0;1) \ddd(2,0) \xx(4,0;1) \draw[thick]
    (4.1,0)--(4.9,0); \oo(5,0;1)
  \end{tikzpicture}
  \\

  \PGL(4)&\PSp(4)&&
  \begin{tikzpicture}[scale=\scalar]
    \xx(0,0;1) \oo(1,0;2) \xx(2,0;1) \draw[thick] (0.1,0)--(0.9,0);
    \draw[thick] (1.1,0)--(1.9,0);
  \end{tikzpicture}\\
  \SO(2n+1)&\<s_{\a_n}\>\cdot\SO(2n)&n\ge2&
  \begin{tikzpicture}[scale=\scalar]
    \oo(0,0;1) \draw[thick] (0.1,0)--(0.9,0); \xx(1,0;1) \draw[thick]
    (1.1,0)--(1.9,0); \xx(2,0;1) \ddd(2,0) \xx(4,0;1) \draw[thick]
    (4.1,0.05)--(4.85,0.05); \draw[thick] (4.1,-0.05)--(4.85,-0.05);
    \xx(5,0;1) \draw[thick] (4.9,0) -- (4.7,0.15); \draw[thick]
    (4.9,0) -- (4.7,-0.15);
  \end{tikzpicture}[\times2]_{p\ne2}
  \\

  \SO(2n+1)&P_n(\SO(2n))&n\ge2&
  \begin{tikzpicture}[scale=\scalar]
    \oo(0,0;1) \draw[thick] (0.1,0)--(0.9,0); \xx(1,0;1) \draw[thick]
    (1.1,0)--(1.9,0); \xx(2,0;1) \ddd(2,0) \xx(4,0;1) \draw[thick]
    (4.1,0.05)--(4.85,0.05); \draw[thick] (4.1,-0.05)--(4.85,-0.05);
    \oo(5,0;1) \draw[thick] (4.9,0) -- (4.7,0.15); \draw[thick]
    (4.9,0) -- (4.7,-0.15);
  \end{tikzpicture}
  \\

  \SO(7)&\G_2&&
  \begin{tikzpicture}[scale=\scalar]
    \xx(0,0;1) \draw[thick] (0.1,0)--(0.9,0); \xx(1,0;2) \draw[thick]
    (1.1,0.05)--(1.85,0.05); \draw[thick] (1.1,-0.05)--(1.85,-0.05);
    \oo(2,0;3) \draw[thick] (1.9,0) -- (1.7,0.15); \draw[thick]
    (1.9,0) -- (1.7,-0.15);
  \end{tikzpicture}
  \\

  \PSp(2n)&\Sp(2)\cdot\Sp(2n-2)&n\ge3&
  \begin{tikzpicture}[scale=\scalar]
    \xx(0,0;1) \draw[thick] (0.1,0)--(0.9,0); \oo(1,0;2) \draw[thick]
    (1.1,0)--(1.9,0); \xx(2,0;2) \ddd(2,0) \xx(4,0;2) \draw[thick]
    (4.15,0.05)--(4.9,0.05); \draw[thick] (4.15,-0.05)--(4.9,-0.05);
    \xx(5,0;1) \draw[thick] (4.1,0) -- (4.3,0.15); \draw[thick]
    (4.1,0) -- (4.3,-0.15);
  \end{tikzpicture}
  \\

  \PSp(2n)&P_1(\Sp(2))\cdot\Sp(2n-2)&n\ge3&
  \begin{tikzpicture}[scale=\scalar]
    \oo(0,0;1) \draw[thick] (0.1,0)--(0.9,0); \oo(1,0;2) \draw[thick]
    (1.1,0)--(1.9,0); \xx(2,0;2) \ddd(2,0) \xx(4,0;2) \draw[thick]
    (4.15,0.05)--(4.9,0.05); \draw[thick] (4.15,-0.05)--(4.9,-0.05);
    \xx(5,0;1) \draw[thick] (4.1,0) -- (4.3,0.15); \draw[thick]
    (4.1,0) -- (4.3,-0.15);
  \end{tikzpicture}
  \\

  \PSO(2n)&\SO(2n-1)&n\ge4&
  \begin{tikzpicture}[scale=\scalar]
    \oo(0,0;2) \draw[thick] (0.1,0)--(0.9,0); \xx(1,0;2) \ddd(1,0)
    \xx(3,0;2) \draw[thick] (3.1,0)--(3.9,0); \xx(4,0;2)
    \xx(4.5,0.5;1) \xx(4.5,-0.5;1) \draw[thick] (4.0,0) -- (4.5,0.5);
    \draw[thick] (4.0,0) -- (4.5,-0.5);
  \end{tikzpicture}
  \\

  \F_4&\Spin(9)&&
  \begin{tikzpicture}[scale=\scalar]
    \xx(0,0;1) \draw[thick] (0.1,0)--(0.9,0); \xx(1,0;2) \draw[thick]
    (1.1,0.05)--(1.85,0.05); \draw[thick] (1.1,-0.05)--(1.85,-0.05);
    \xx(2,0;3) \draw[thick] (1.9,0) -- (1.7,0.15); \draw[thick]
    (1.9,0) -- (1.7,-0.15); \oo(3,0;2) \draw[thick] (2.1,0)--(2.9,0);
  \end{tikzpicture}
  \\

  \G_2&\<s_{\a_1}\>\cdot\SL(3)&&
  \begin{tikzpicture}[scale=\scalar]
    \oo(4,0;2) \draw[thick] (4.2,0.07)--(4.9,0.07); \draw[thick]
    (4.1,0)--(4.9,0); \draw[thick] (4.2,-0.07)--(4.9,-0.07);
    \xx(5,0;1) \draw[thick] (4.1,0) -- (4.3,0.15); \draw[thick]
    (4.1,0) -- (4.3,-0.15);
  \end{tikzpicture}[\times2]_{p\ne2}
  \\

  \G_2&\multicolumn2l{\GL(2)_{\rm long}U_{2\a_1{+}\a_2}U_{3\a_1{+}\a_2}U_{3\a_1{+}2\a_2}}&
  \begin{tikzpicture}[scale=\scalar]
    \oo(0,0;1) \draw[thick] (0.2,0.07)--(0.9,0.07); \draw[thick]
    (0.1,0)--(0.9,0); \draw[thick] (0.2,-0.07)--(0.9,-0.07);
    \oo(1,0;1) \draw[thick] (0.1,0) -- (0.3,0.15); \draw[thick]
    (0.1,0) -- (0.3,-0.15);
  \end{tikzpicture}

  \\
  \noalign{\smallskip\hrule}
  \PGL(2)\times\PGL(2)&(\|id|\times F_q)\PGL(2)&&
  \begin{tikzpicture}[scale=\scalar]
    \oo(0,0;q) \oo(1,0;1)
    \draw[thick](0,-0.1)--(0,-0.3)--(1,-0.3)--(1,-0.1);
  \end{tikzpicture}
  \\
  \multicolumn{3}{r}{\text{$F_q$=Frobenius morphism, $q=p^l\ge1$}}&\\

  \noalign{\smallskip\hrule\smallskip}

  \multicolumn4l{\text{Additionally for $p=2$:}}\\

  \SO(2n+1)&\SO(3)\times\SO(2n-1)&n\ge3&
  \begin{tikzpicture}[scale=\scalar]
    \xx(0,0;1) \draw[thick] (0.1,0)--(0.9,0); \oo(1,0;2) \draw[thick]
    (1.1,0)--(1.9,0); \xx(2,0;2) \ddd(2,0) \xx(4,0;2) \draw[thick]
    (4.1,0.05)--(4.85,0.05); \draw[thick] (4.1,-0.05)--(4.85,-0.05);
    \xx(5,0;2) \draw[thick] (4.9,0) -- (4.7,0.15); \draw[thick]
    (4.9,0) -- (4.7,-0.15);
  \end{tikzpicture}
  \\

  \SO(2n+1)&P_1(\SO(3))\times\SO(2n-1)&n\ge3&
  \begin{tikzpicture}[scale=\scalar]
    \oo(0,0;1) \draw[thick] (0.1,0)--(0.9,0); \oo(1,0;2) \draw[thick]
    (1.1,0)--(1.9,0); \xx(2,0;2) \ddd(2,0) \xx(4,0;2) \draw[thick]
    (4.1,0.05)--(4.85,0.05); \draw[thick] (4.1,-0.05)--(4.85,-0.05);
    \xx(5,0;2) \draw[thick] (4.9,0) -- (4.7,0.15); \draw[thick]
    (4.9,0) -- (4.7,-0.15);
  \end{tikzpicture}
  \\

  \PSp(2n)&\<s_{\a_n}\>\cdot\PSO(2n)&n\ge2&
  \begin{tikzpicture}[scale=\scalar]
    \oo(0,0;2) \draw[thick] (0.1,0)--(0.9,0); \xx(1,0;2) \draw[thick]
    (1.1,0)--(1.9,0); \xx(2,0;2) \ddd(2,0) \xx(4,0;2) \draw[thick]
    (4.15,0.05)--(4.9,0.05); \draw[thick] (4.15,-0.05)--(4.9,-0.05);
    \xx(5,0;1) \draw[thick] (4.1,0) -- (4.3,0.15); \draw[thick]
    (4.1,0) -- (4.3,-0.15);
  \end{tikzpicture}
  \\

  \PSp(2n)&P_n(\PSO(2n))&n\ge2&
  \begin{tikzpicture}[scale=\scalar]
    \oo(0,0;2) \draw[thick] (0.1,0)--(0.9,0); \xx(1,0;2) \draw[thick]
    (1.1,0)--(1.9,0); \xx(2,0;2) \ddd(2,0) \xx(4,0;2) \draw[thick]
    (4.15,0.05)--(4.9,0.05); \draw[thick] (4.15,-0.05)--(4.9,-0.05);
    \oo(5,0;1) \draw[thick] (4.1,0) -- (4.3,0.15); \draw[thick]
    (4.1,0) -- (4.3,-0.15);
  \end{tikzpicture}
  \\

  \PSp(6)&\G_2&&
  \begin{tikzpicture}[scale=\scalar]
    \xx(0,0;2) \draw[thick] (0.1,0)--(0.9,0); \xx(1,0;4) \draw[thick]
    (1.15,0.05)--(1.9,0.05); \draw[thick] (1.15,-0.05)--(1.9,-0.05);
    \oo(2,0;3) \draw[thick] (1.1,0) -- (1.3,0.15); \draw[thick]
    (1.1,0) -- (1.3,-0.15);
  \end{tikzpicture}
  \\

  \F_4&\Sp(8)&&
  \begin{tikzpicture}[scale=\scalar]
    \oo(0,0;2) \draw[thick] (0.1,0)--(0.9,0); \xx(1,0;3) \draw[thick]
    (1.1,0.05)--(1.85,0.05); \draw[thick] (1.1,-0.05)--(1.85,-0.05);
    \xx(2,0;4) \draw[thick] (1.9,0) -- (1.7,0.15); \draw[thick]
    (1.9,0) -- (1.7,-0.15); \xx(3,0;2) \draw[thick] (2.1,0)--(2.9,0);
  \end{tikzpicture}
  \\

  \noalign{\smallskip\hrule\smallskip}

  \multicolumn4l{\text{Additionally for $p=3$:}}\\

  \G_2&\<s_{\a_2}\>\cdot\SL(3)_{\rm short}&&
  \begin{tikzpicture}[scale=\scalar]
    \xx(0,0;3) \draw[thick] (0.2,0.07)--(0.9,0.07); \draw[thick]
    (0.1,0)--(0.9,0); \draw[thick] (0.2,-0.07)--(0.9,-0.07);
    \oo(1,0;2) \draw[thick] (0.1,0) -- (0.3,0.15); \draw[thick]
    (0.1,0) -- (0.3,-0.15);
  \end{tikzpicture}[\times2]
  \\

  \G_2&\multicolumn2l{\GL(2)_{\rm short}U_{\a_1{+}\a_2}U_{2\a_1{+}\a_2}U_{3\a_1{+}2\a_2}}&
  \begin{tikzpicture}[scale=\scalar]
    \oo(0,0;3) \draw[thick] (0.2,0.07)--(0.9,0.07); \draw[thick]
    (0.1,0)--(0.9,0); \draw[thick] (0.2,-0.07)--(0.9,-0.07);
    \oo(1,0;1) \draw[thick] (0.1,0) -- (0.3,0.15); \draw[thick]
    (0.1,0) -- (0.3,-0.15);
  \end{tikzpicture}
\end{array}
$$

\end{document}